\documentclass[12pt,oneside,a4paper]{article}

\usepackage[english]{babel}
\usepackage{amsmath,amssymb,amsfonts}
\usepackage[gen,right]{eurosym}
\usepackage{textcomp,longtable,epsfig,url}
\usepackage{bbm}
\usepackage{graphicx}
\usepackage{subfigure}
\usepackage{hyperref}
\usepackage[intoc]{nomencl}
\usepackage[printonlyused]{acronym}
\usepackage{verbatim}
\usepackage{tabularx}
\usepackage{url}
\usepackage{color}
\usepackage{amssymb,amsmath}
\usepackage{setspace}
\usepackage{scrhack}
\usepackage{floatrow}
\usepackage{listings}
\usepackage{amsthm}
\usepackage{paralist}
\usepackage{xr}
\usepackage{graphicx}
\usepackage{listings}
\usepackage{booktabs}
\usepackage{flafter}
\usepackage{varwidth}
\usepackage[round]{natbib}
\usepackage[a4paper]{geometry}
\usepackage{enumitem}
\usepackage{titling}

\makeindex

\newtheorem{definition}{Definition}[section]

\newtheorem{theorem}[definition]{Theorem}

\newcommand{\beqo}{\begin{eqnarray*}}
	\newcommand{\eeqo}{\end{eqnarray*}\noindent}
\newcommand{\beq}{\begin{eqnarray}}
	\newcommand{\eeq}{\end{eqnarray}\noindent}

\newcommand{\VaR}{\mathrm{VaR}}

\title{Weighted universal Value-at-Risk Superadditivity for discrete distributions}

\author{ Alfred M\"uller\\ Department Mathematik\\ University of Siegen, Germany\\ \texttt{mueller@mathematik.uni-siegen.de}}

\date{\today.\\[.3cm]
	}

\begin{document}
\begin{titlingpage}
		\maketitle
\begin{abstract}
The concept of weighted universal Value-at-Risk superadditivity (WUVS) was recently introduced by  \cite{chen2026universal} as a generalization of the question whether for some 
 infinite mean distributions convex combinations of i.i.d. random variables can stochastically dominate the parent distribution. In this short note we prove that the property WUVS can basically never hold for discrete distributions except for the case of comonotonicity. This implies as a corollary that for discrete distributions with infinite mean it can also never hold that convex combinations of i.i.d. random variables can stochastically dominate the parent distribution. This settles an open problem mentioned in \cite{muller2025}.

~\\			
			
{\raggedleft {\bf Keywords:} risk sharing, infinite mean models, Value-at-Risk superadditivity}   \\[0.2cm]
					\end{abstract}
	\end{titlingpage}

	\pagestyle{plain}
	
	\onehalfspacing
	\allowdisplaybreaks

\section{Introduction}

In a paper by \cite{chen-etal2024} it has been shown that Pareto distributed random variables with infinite mean fulfill the property that $X_1$ is stochastically dominated by any convex combination of i.i.d. copies of it.  Motivated by this result several other studies have followed that consider the question, how general this result holds. We will denote as in \cite{muller2025} the class of distributions that fulfill this property as $\mathcal{D}^-$. Notice that this property can be interpreted as a negative diversification effect when considering a portfolio of financial risks. 

We collect here a few of the recent results on this topic. 
\cite{vincent2025}  considers distributions that he calls subscalable. This property holds for a random variable $X$ with cdf $F$, if $x \mapsto x (1-F(x))$ is increasing on $(0,\infty)$.
\cite{muller2025} shows that the property holds for distributions of random variables that are convex transformations of Cauchy distributions. Further related results have been shown in \cite{arab2025ecp}, \cite{chen-shneer} and \cite{zeng2026}, among others. These distributions typically are continuous. So far there does not exist an example with this property of a discrete distribution assuming only finite values. \cite{muller2025} mentions the case of two-point distributions with values in $\{0, \infty\}$ and \cite{arab2025ecp} give an example with values in $\{1,2, \ldots, \infty\}$ in their Example 2.7, where also explicitly $P(X = \infty) > 0$. \cite{arab2025ecp} also mention in their Proposition 2.8 that their sufficient condition for being in $\mathcal{D}^-$ can hold for a discrete distribution only if $P(X = \infty) > 0$, whereas they mention in their introduction that it is reasonable to expect that the property may hold even for some discrete
models. \cite{vincent2025} shows that random variables with an infinite mean discrete Pareto distribution are stochastically dominated by equally weighted convex combinations of i.i.d. copies, but he mentions that this may fail for arbitrary convex combinations. 

Thus in all these recent papers there is no example of an infinite mean distribution in $\mathcal{D}^-$ that is discrete with values in a countable subset of the real line. It was mentioned explicitly as an open problem in \cite{muller2025},
whether there is a discrete distribution in $\mathcal{D}^-$ that assumes only finite values.

In a recent paper by \cite{chen2026universal} the concept of \emph{weighted universal Value-at-Risk superadditivity (WUVS)}  was introduced. This is a generalization of the property of a distribution to be in $\mathcal{D}^-$ to random vectors with components that are not necessarily i.i.d.. It  compares Value-at-Risks of convex combinations to convex combinations of Value-at-Risks. In this paper we will show that the property WUVS can basically never hold for classical discrete distributions on a countable set of values except for the case of comonotonicity. This will imply as a corollary that a discrete non-degenerate distribution can not be in $\mathcal{D}^-$, thus solving the open problem posed in \cite{muller2025}.  

{\bf Notation:} Throughout the paper we will assume the existence of a probability space $(\Omega, \mathcal{A}, P)$, on which we can define random variables $X:\Omega \to \mathbb{R}$ with arbitrary distributions. In this paper we will often consider discrete random variables that fulfill the following condition of discreteness. 

\begin{definition}  \label{def-discrete}
We say that a random variable is \textbf{discrete}, if there is 
are a countable index set $I \subset \mathbb{Z}$  and a set of numbers $\{x_i: i \in I\} \subseteq \mathbb{R}$ with $x_i < x_j$ for $i < j$, such that $P(X = x_i) > 0$ for all $i \in I$ and $\sum_{i \in I} P(X = x_i) = 1$. 
\end{definition}

We allow for arbitrary real values of $X$ even though in many applications we may have the special case of discrete random variables with non-negative integer values $x_i := i, \ i \in \mathbb{N}_0$. 
We will denote by 
$
F_X(t) := P(X\le t), \ t \in \mathbb{R},
$
the corresponding cumulative distribution function of $X$, and we will also call $F_X$ just the distribution of $X$. For any random variable $X$ with cumulative distribution function $F$ we define the Value-at-Risk of $X$ at level $p \in (0,1)$  in terms of the  generalized inverse of $F$ as follows:
$$
 \VaR_{p}(X) =F^{-1}(p) := \inf\{x \in \mathbb{R}: F(x) \ge p\}, \ 0 < p < 1.
$$
We write $X =_{st} Y$ if $F_X = F_Y$ and we define the usual stochastic order by
$$
X \le_{st}  Y \ \Leftrightarrow \ F_X(t) \ge F_Y(t) \mbox{ for all } t \in \mathbb{R}  \ \Leftrightarrow \ F_X^{-1}(p) \le F_Y^{-1}(p) \mbox{ for all } p \in (0,1). 
$$
This usual stochastic order $\le_{st}$ can also be found under the name \textit{first order stochastic dominance} in an economic context and it is generally agreed that any rational decision will prefer a risk $X$ to a risk $Y$, if $X \le_{st}  Y$ holds. It is the strongest of the well known stochastic dominance rules used for decisions under risk and in contrast to weaker notions like the second order stochastic dominance it is well defined for arbitary distributions, including infinite mean models. For properties of this and related stochastic orders we refer to the books \cite{muesto2002} and \cite{ShaSha:Springe2007}.

\newpage

\section{Main results}

\cite{chen2026universal}  introduce the following concept of 
 \emph{weighted universal Value-at-Risk superadditivity (WUVS)} for a random vector $(X_1, \dots, X_n)$ if it holds
\begin{equation} \label{wuvs-def}
   \VaR_{p}\left( \sum_{i=1}^n \theta_iX_{i}\right) \ge  \sum_{i=1}^n\VaR_{p}(\theta_iX_i) \mbox{~for all~} p\in(0,1) \mbox{~and~} (\theta_1,\dots,\theta_n) \in\Delta_n,
\end{equation} 
where $\Delta_n=\{(\theta_1,\dots,\theta_n)\in [0,1]^n:  \sum_{i=1}^n \theta_i=1\}$.  It is shown in Example 2 of  \cite{chen2026universal} that for i.i.d. random variables with the discrete distributions from the St. Petersburg lottery with $P(X = 2^k) = 2^{-k}, \ k \in \mathbb{N}$, the property WUVS does not hold, whereas in that case the weaker property UVS holds, i.e. the condition in  \eqref{wuvs-def}  holds under the additional assumption $\theta_1 = \ldots = \theta_n$. As our main result we show here the following very general negative result for discrete distributions, allowing for almost arbitrary discrete marginal distributions and dependence structures.

\begin{theorem} \label{theorem:wuvs}
Assume that the random variables $X,Y$ are discrete as defined in Definition \ref{def-discrete}.  Assume that there are indices $i,j$ with  $P(X=x_{i+1},Y=y_j) > 0$ and $P(X=x_i,Y=y_{j+1}) > 0$, then the random vector $(X,Y)$ does not have the property WUVS. 
\end{theorem}

\begin{proof}
It has been shown in Theorem 1 of \cite{chen2026universal} that WUVS is preserved if we replace $X,Y$ by increasing convex transformations $f(X), g(Y)$. Assuming that $(X,Y)$ has the WUVS property, the same therefore holds for $(\max\{X,x_i\},\max\{Y,y_j\} )$.  Morover, the property is not affected by linear increasing transformations and thus can assume without loss of generality that $X$ assumes only values $x_0 < x_1 < \ldots$ and $Y$ assumes only values $y_0 < y_1 < \ldots$ and that the first two values are $x_0 = y_0 = 0$ and $x_1 = y_1 = 1$. We can also assume without loss of generality that $P(X=0) \ge P(Y = 0)$ as we otherwise can exchange $X$ and $Y$. We will show now that under this assumption the vector $(X,Y)$ can not have the property WUVS. Choose some $p$ with  $P(Y = 0) \le P(X = 0)  < p <  P(X = 0) + P(X=1,Y= 0)$. Then we get for some small enough $\varepsilon >0$ that
$$
P(((1-\varepsilon)X + \varepsilon Y \le  1-\varepsilon/2) \ge P(X=0, Y \le 1/\varepsilon -1/2) + P(X=1,Y=0) > p
$$
and therefore 
$\VaR_{p}((1-\varepsilon)X + \varepsilon Y)  \le  1-\varepsilon/2$. As $P(Y = 0) \le P(X = 0)  < p$ we get $\VaR_{p}(X)\ge 1$ and  $\VaR_{p}(Y)  \ge 1$ and this implies
$$
\VaR_{p}((1-\varepsilon)X )  + \VaR_{p}(\varepsilon Y )  = (1-\varepsilon) \VaR_{p}(X) +  \varepsilon  \VaR_{p}(Y)  \ge  1 > \VaR_{p}((1-\varepsilon)X + \varepsilon Y) .
$$
Hence the random vector $(X,Y)$ does not have the property WUVS. 
\end{proof}

Notice that it has been shown in Proposition 1 of \cite{chen2026universal} that WUVS can hold for random variables with a finite mean only if they are comonotonic. The assumption $P(X=x_{i+1},Y=y_j) > 0$ and $P(X=x_i,Y=y_{j+1}) > 0$ in Theorem \ref{theorem:wuvs} clearly contradicts the assumption of comonotonicity, and the theorem thus says that under this very weak condition it is impossible for a discrete distribution to fulfill WUVS even in case of infinite means without any further assumptions on the marginals and the dependence. 

This is a bit surprising, as there are many examples of continuous distributions with the property WUVS holding with strict inequality. Theorem \ref{theorem:wuvs} shows, however, so that no fine discretization does preserve this property. This means in particular that the set of distributions with the property WUVS is nowhere dense with respect to weak convergence. On the other hand, the property WUVS is preserved under convergence in distribution as shown in Theorem 3 of \cite{chen2026universal}.

The notion of WUVS is a generalization of the following concept $\mathcal{D}^-$ defined in \cite{muller2025} and used frequently in the recent literature. 
	
\begin{definition}
Let $\mathcal{D}^-$ denote the set of distribution functions $F_X$ with the property that
\begin{equation} \label{def-d}
X_1 \le_{st} \sum_{i=1}^n \theta_i X_i,
\end{equation}
for all $(\theta_1,\dots,\theta_n) \in\Delta_n$, where $X_1,\ldots,X_n$ are i.i.d. random variables with distribution functions $F_X$. We will also write $X \in \mathcal{D}^-$, if  $F_X \in \mathcal{D}^-$.
\end{definition}

The property  $X \in \mathcal{D}^-$ means in terms of risks that diversification of a portfolio of independent such risks has a negative effect by stochastically increasing the risk. This property may be a bit surprising at first sight, as one typically assumes that diversification is good at least for risk averse decision makers. A detailed discussion of this property in the context of risk exchange is given in \cite{chen2024risk}.   

For identically distributed $X_1, \ldots, X_n$ we get
$
 \sum_{i=1}^n\VaR_{p}(\theta_iX_i) = VaR_{p}(X_1) 
$ 
and this immediately implies that a vector $X = (X_1, \ldots, X_n)$ with i.i.d. components is WUVS, if and only $X_1 \in \mathcal{D}^-$. 

\begin{theorem} \label{th:main}
Assume that $X$ is non-degenerate and discrete as defined in Definition \ref{def-discrete}.  Then $X \not\in \mathcal{D}^-$.
\end{theorem}

\begin{proof}
If $X,Y$ are  i.i.d. with the same distribution as $X$, then for any two points $x_i < x_{i+1}$ in the support we get $P(X = x_i, Y = x_{i+1}) =  P(X = x_{i+1}, Y = x_i) = P(X=x_{i+1})P(Y=x_{i}) > 0$, and  therefore the result is an immediate consequence of Theorem \ref{theorem:wuvs}.  
\end{proof}

\bibliographystyle{apalike}


\end{document}